\documentclass{amsart}
\usepackage{amsmath,amsthm,amssymb,amscd}
\usepackage{wrapfig}

\newtheorem{thm}{Theorem}[section]
\newtheorem{prop}[thm]{Proposition}
\newtheorem{lem}[thm]{Lemma}
\newtheorem{cor}[thm]{Corollary}

\theoremstyle{definition}

\theoremstyle{definition}
\newtheorem{eg}[thm]{Example}

\theoremstyle{remark}

\newcommand{\Hom}{{\rm Hom}}

\newcommand{\Z}{\mathbb{Z}}

\newcommand{\R}{\mathbb{R}}

\title{Homotopy types of the Hom complexes of graphs}

\begin{document}

\author{Takahiro Matsushita}
\address{Department of Mathematics, Kyoto University, KitaShirakawa Oiwake-cho Sakyo-ku, Kyoto 606-8502 Japan}
\email{mtst@math.kyoto-u.ac.jp}

\begin{abstract}
The Hom complex $\Hom(T,G)$ of graphs is a CW-complex associated to a pair of graphs $T$ and $G$, considered in the graph coloring problem. It is known that certain homotopy invariants of $\Hom(T,G)$ give lower bounds for the chromatic number of $G$.

For a fixed finite graph $T$, we show that there is no homotopy invariant of $\Hom(T,G)$ which gives an upper bound for the chromatic number of $G$. More precisely, for a non-bipartite graph $G$, we construct a graph $H$ such that $\Hom(T,G)$ and $\Hom(T,H)$ are homotopy equivalent but $\chi(H)$ is much larger than $\chi(G)$. The equivariant homotopy type of $\Hom(T,G)$ is also considered.
\end{abstract}

\maketitle

\section{Introduction}


The application of algebraic topology to the graph coloring problem started from Lov\'asz's proof \cite{Lovasz} of the Kneser conjecture. Lov\'asz introduced the neighborhood complex $N(G)$ of a graph $G$ and showed that if $N(G)$ is $n$-connected, then the chromatic number $\chi(G)$ of $G$ is greater than $n + 2$. 

The Hom complex $\Hom(T,G)$ is a CW-complex associated to a pair of graphs $T$ and $G$, considered in the context of the graph coloring problem \cite{BK1}. It is known that the neighborhood complex $N(G)$ and $\Hom(K_2,G)$ are homotopy equivalent.

A graph $T$ is a {\it homotopy test graph} \cite{Kozlov book} if the inequality
\begin{eqnarray}
\chi(G) > {\rm conn}(\Hom(T,G)) + \chi(T)
\end{eqnarray}
holds for every graph $G$. Here for a space $X$, ${\rm conn}(X)$ denotes the largest integer $n$ such that $X$ is $n$-connected. Thus Lov\'asz's result implies that $K_2$ is a homotopy test graph. Other examples of homotopy test graphs are given by complete graphs $K_n$ with $n \geq 3$ \cite{BK1}, odd cycles $C_{2r+1}$ \cite{BK2}, bipartite graphs \cite{Matsushita 1}, and some of the stable Kneser graphs \cite{Schultz}. For further development and related topics, we refer to \cite{DS}, \cite{HL}, and \cite{Kozlov book}.

Here we give a brief explanation of how to obtain the lower bound for $\chi(G)$ from $\Hom(T,G)$. Consider a group action on the graph $T$. In many cases, we consider a cyclic group $\Z_2$ of order 2 and a $\Z_2$-action flipping some edge of $T$. Then this action on $T$ induces a group action on $\Hom(T,G)$, and the graph homomorphism $f: G_1 \rightarrow G_2$ induces an equivariant map from $\Hom(T,G_1)$ to $\Hom(T,G_2)$. Thus if there is no equivariant map from $\Hom(T,G)$ to $\Hom(T,K_n)$, then $\chi(G)$ is greater than $n$, and we have a lower bound.

Therefore it is natural to ask that the chromatic number of $G$ is determined by the (equivariant) homotopy type of $\Hom(T,G)$. There are several works concerning this question. Walker \cite{Walker} showed that there are graphs $G_1$ and $G_2$ such that $\Hom(K_2,G_1)$ and $\Hom(K_2,G_2)$ are $\Z_2$-homotopy equivalent but their chromatic numbers are different (see Section 12 of \cite{Walker} or Example \ref{Example 3.3}). The author constructed graphs $H_1$ and $H_2$ that $\Hom(K_2,G_1)$ and $\Hom(K_2,G_2)$ are homeomorphic but their chromatic numbers are different. Kozlov \cite{Kozlov 2} considered simple tests for the chromatic numbers, using the homology groups of the Hom complexes (homology tests), and showed that there are differences between these tests and the actual chromatic numbers.

Now we state the main result in this paper. This is a further generalization of Walker's example.

\begin{thm} \label{Theorem 1.1}
Let $T$ be a finite graph and $G$ a non-bipartite graph. Then for each integer $n$, there is a graph $H$ such that $\Hom(T,G)$ and $\Hom(T,H)$ are homotopy equivalent but $\chi(H) > n$. Moreover, if $G$ has no looped vertices, then neither has $H$.
\end{thm}

When $T$ is a $\Z_2$-graph and the involution flips some edge of $T$, we can choose $H$ so that $\Hom(T,G)$ and $\Hom(T,H)$ are $\Z_2$-homotopy equivalent. This theorem is deduced from Theorem \ref{Theorem 5.1} and Corollary \ref{Corollary 5.5}.

Theorem \ref{Theorem 1.1} implies that there is no homotopy invariant of $\Hom(T,G)$ which gives an upper bound for the chromatic number of $G$. In particular, the homotopy type of $\Hom(T,G)$ does not determine $\chi(G)$ although $T$ is a homotopy test graph. In particular, consider the case $T = K_2$. For a $\Z_2$-space $X$, we write ${\rm ind}_{\Z_2}(X)$ to indicate the smallest integer $k$ such that there is a $\Z_2$-map from $X$ to $S^k$. Since $\Hom(K_2,K_n)$ is $\Z_2$-homeomorphic to $S^{n-2}$ (see \cite{BK1}), we have the inequality
\begin{eqnarray} \label{eqn 1.1}
\chi(G) \geq {\rm ind}_{\Z_2}(\Hom( K_2,G)) + 2
\end{eqnarray}
for every graph $G$. It is known that the right of the inequality (\ref{eqn 1.1}) is the largest lower bound obtained from the $\Z_2$-homotopy type of $\Hom(K_2,G)$ (see Theorem 1.6 and Theorem 1.7 of \cite{DS}). By Theorem \ref{Theorem 1.1} we cannot obtain more information than the inequality (\ref{eqn 1.1}) from the $\Z_2$-homotopy type of $\Hom(K_2,G)$.

To explain the outline of the proof, we recall the following theorem. Here we denote by $g(G)$ the girth of $G$.

\begin{thm}[Erd\H{o}s \cite{Erdos}] \label{Theorem 1.2}
For a pair of positive integers $m$ and $n$, there is a finite graph $G$ such that $\chi(G) > m$ and $g(G) > n$.
\end{thm}


We now explain the outline of the proof of Theorem \ref{Theorem 1.1}. By Theorem \ref{Theorem 1.2}, there is a graph $X$ such that both of the girth and the chromatic number of $X$ are quite large. By the results of Section 3 and Section 4, we have the following: There is a graph $Y$, and graph homomorphisms $f:Y \rightarrow X$ and $g: Y \rightarrow G$ such that $f$ induces a homotopy equivalence $f_* : \Hom(T,Y) \rightarrow \Hom(T,X)$. Next we consider the ``cylinder" $Y \times I_k$ of $Y$ (see Section 2.3) for a sufficiently large integer $k$. Let $H$ be the graph obtained by attaching two ends of $Y \times I_k$ by the graph homomorphisms $f: Y \rightarrow X$ and $g:Y \rightarrow G$.

The graph $H$ has desired properties. In fact, since $X$ is a subgraph of $H$, we have $\chi(H) \geq \chi(X) > n$. The reader who is familiar with algebraic topology may notice that this construction is similar to the homotopy pushouts of spaces. In fact it turns out that $\Hom(T,H)$ is the homotopy pushout of $f_* : \Hom(T,Y) \rightarrow \Hom(T,X)$ and $g_* : \Hom(T,Y) \rightarrow \Hom(T,G)$. Since $f_*$ is a homotopy equivalence, $\Hom(T,H)$ is homotopy equivalent to $\Hom(T,G)$. This is the outline of the proof.

\vspace{2mm} \noindent {\bf Acknowledgement.} The author wishes to express his gratitude to Toshitake Kohno for his encouragement and helpful comments. The author thanks Dmitry N. Kozlov, Shouta Tounai, and the anonymous referees for insightful comments and valuable suggestions. The author is supported by the Grant-in-Aid for Scientific Research (KAKENHI No. ~28-6304) and the Grant-in-Aid for JSPS fellows.

\section{Preliminaries}

In this section we review relevant facts and introduce the terminology. For a concrete introduction to this subject, we refer to Kozlov's book \cite{Kozlov book}.

Let $P$ be a poset. The {\it order complex $\Delta(P)$ of $P$} is the abstract simplicial complex whose vertex set is the underlying set of $P$ and whose simplices are finite chains of $P$. The {\it classifying space of $P$} is the geometric realization of $\Delta(P)$, and is denoted by $|P|$. We often identify the poset $P$ with the space $|P|$, and assign topological terminology to posets by the classifying space functor. For example, an order-preserving map $f:P \rightarrow Q$ is a homotopy equivalence if the continuous map $|f| : |P| \rightarrow |Q|$ induced by $f$ is a homotopy equivalence.

\subsection{Graphs}

A {\it graph} is a pair $G = (V(G),E(G))$ consisting of a set $V(G)$ together with a symmetric subset $E(G)$ of $V(G) \times V(G)$. Hence our graphs are undirected, may have loops, but have no multiple edges. The reason why we admit looped vertices will be found in Section 2.3 (see the definition of $I_n$). For a pair of vertices $v$ and $w$, we write $v \sim w$ if $(v,w)$ belongs to $E(G)$. A graph is {\it simple} if it has no looped vertices. For a pair of vertices $x$ and $w$ of $G$, we write $\langle x, w\rangle$ to indicate the subset $\{ (x,w), (w,x)\}$ of $V(G) \times V(G)$. The {\it neighborhood $N(v)$ of a vertex $v$} is the set of vertices adjacent to $v$. A {\it graph homomorphism} is a map $f: V(G) \rightarrow V(H)$ with $(f \times f)(E(G)) \subset E(H)$.

For a non-negative integer $n$, let $K_n$ be the complete graph with $n$-vertices, namely, $V(K_n) = \{ 1,\cdots, n\}$ and $E(K_n) = \{ (x,y) \; | \; x \neq y\}$. In terms of these notions, an $n$-coloring of $G$ is identified with a graph homomorphism from $G$ to $K_n$. The {\it chromatic number of $G$} is the number
$$\chi(G) = \inf \{ n \geq 0 \; | \; \textrm{There is an $n$-coloring of $G$.}\}.$$
Here we consider the infimum of the empty set is $+ \infty$.

The {\it girth of a graph $G$} is the minimal length of cycles embedded into $G$.

\subsection{Hom complex}

A {\it multi-homomorphism from $G$ to $H$} is a map $\eta : V(G) \rightarrow 2^{V(H)} \setminus \{ \emptyset \}$ such that $(v,w) \in E(G)$ implies $\eta(v) \times \eta(w) \subset E(H)$. We write $\eta \leq \eta'$ if $\eta(v) \subset \eta'(v)$ for all $v \in V(G)$. The {\it Hom complex $\Hom(G,H)$} is the poset of multi-homomorphisms from $G$ to $H$ with the above ordering. Note that a graph homomorphism $f: G \rightarrow H$ is identified with a minimal point of $\Hom(G,H)$.

We compare the definition of Hom complexes with others. Our definition of the Hom complex is due to Dochtermann \cite{Dochtermann}. In \cite{BK1}, Babson and Kozlov give the definition of the Hom complex as a certain subcomplex of products of simplices in case $T$ and $G$ are finite. Then our Hom complex is isomorphic to the face poset of their Hom complex. Thus the topological types of the two definitions coincide.

For a graph homomorphism $f: G_1 \rightarrow G_2$, let $f^* : \Hom(G_2, H) \rightarrow \Hom(G_1,H)$ be the order-preserving map defined by $f^* (\eta) = \eta \circ f$. On the other hand, for a graph homomorphism $g: H_1 \rightarrow H_2$, let $g_* : \Hom(G,H_1) \rightarrow \Hom(G,H_2)$ be the order-preserving map defined by $g_*(\eta)(x) = g(\eta(x))$.

An {\it involution of a graph $T$} is a graph homomorphism $\alpha : T \rightarrow T$ with $\alpha^2 = {\rm id}_T$. An involution is {\it flipping} if there is a vertex $x$ of $T$ such that $(x,\alpha(x)) \in E(T)$. A $\Z_2$-graph $T$ is {\it flipping} if the involution is flipping. It is easy to see that if $T$ is a flipping $\Z_2$-graph and $G$ is simple, then $\Hom(T,G)$ is a free $\Z_2$-space \cite{BK1}. Moreover, the order-preserving map $f_* :\Hom(T,G_1) \rightarrow \Hom(T,G_2)$ induced by a graph homomorphism $f$ is $\Z_2$-equivariant.

\subsection{$\times$-homotopy theory} Here we review the $\times$-homotopy theory established by Dochtermann \cite{Dochtermann} as far as we need.

Let $G$ and $H$ be graphs. Graph homomorphisms $f$ and $g$ from $G$ to $H$ are {\it $\times$-homotopic} if they belong to the same connected component of $\Hom(G,H)$. We write $f \simeq_\times g$ to mean that $f$ and $g$ are $\times$-homotopic. A graph homomorphism $f:G \rightarrow H$ is a {\it $\times$-homotopy equivalence} if there is a graph homomorphism $g: H \rightarrow G$ such that $gf \simeq_\times {\rm id}_G$ and $fg \simeq_\times {\rm id}_H$. A {\it $\times$-homotopy equivalence} is a graph homomorphism $f: G\rightarrow H$ such that there is a graph homomorphism $g:H \rightarrow G$ with $fg \simeq_\times {\rm id}_H$ and $gf \simeq_\times {\rm id}_G$.

Let $a$ and $b$ be integers with $a \leq b$. Let $I_{[a,b]}$ be the graph defined by $V(I_{[a,b]}) = \{ x \in \mathbb{Z} \; | \; a \leq x \leq b\}$ and $E(I_{[a,b]}) = \{ (x,y) \; | \; |x-y| \leq 1\}$. We write $I_n$ instead of $I_{[0,n]}$. For a pair of graph homomorphisms $f,g:G \rightarrow H$, a {\it $\times$-homotopy} from $f$ to $g$ is a graph homomorphism $F: G \times I_n \rightarrow H$ such that $F(x,0) = f(x)$ and $F(x,n) = g(x)$ for all $x \in V(G)$. Here the notation ``$\times$" means the categorical product. Then one can show that $f$ and $g$ are $\times$-homotopic if and only if there is a $\times$-homotopy from $f$ to $g$ (Proposition 4.7 of \cite{Dochtermann}). By this characterization, one can easily show that the inclusion $\iota_c: T \hookrightarrow T \times I_{[a,b]}$, $x \mapsto (x,c)$ for $a \leq c \leq b$ is a $\times$-homotopy equivalence, whose $\times$-homotopy inverse is the projection $p: X \times I_{[a,b]} \rightarrow X$.

Let $T$, $G$, and $H$ be graphs, and let $f$ and $g$ be graph homomorphisms from $G$ to $H$. If $f$ and $g$ are $\times$-homotopic, then the maps $f_*, g_* : \Hom (T,G) \rightarrow \Hom(T,H)$ are homotopic (Theorem 5.1 of Dochtermann \cite{Dochtermann}. However, in his proof, he used Proposition \ref{Proposition 2.3} mentioned below. A direct and simplified proof without using Proposition \ref{Proposition 2.3} was given in Section 5 in \cite{Matsushita 2}). Thus if $f$ is a $\times$-homotopy equivalence, then $f_*$ is a homotopy equivalence.

A vertex $v \in V(G)$ is {\it dismantlable} if there is another vertex $w$ such that $N(v) \subset N(w)$. We write $G \setminus v$ to indicate the induced subgraph of $G$ whose vertex set is $V(G) \setminus \{ v\}$. If $v$ is dismantlable, then the inclusion $G \setminus v \hookrightarrow G$ is a $\times$-homotopy equivalence. In particular, the following lemma holds:

\begin{prop}[Kozlov \cite{Kozlov}] \label{Proposition 2.3}
If $v$ is dismantlable, then the inclusion $\Hom(T,G\setminus v) \hookrightarrow \Hom(T,G)$ is a homotopy equivalence.
\end{prop}

\subsection{Some theorems in algebraic topology}

We need the following theorems.

\begin{thm} \label{Theorem 2.7}
Let $X$ and $Y$ be CW-complexes, and $A$ a set. Let $(X_{\alpha})_{\alpha \in A}$ (or $(Y_\alpha)_{\alpha \in A}$) be an $A$-indexed family of subcomplexes of $X$ (or $Y$) which is a covering of $X$ (or $Y$, respectively). Let $f : X \rightarrow Y$ be a continuous map such that $f(X_\alpha) \subset Y_\alpha$ for every $\alpha \in A$. Suppose that for each finite subset $\{ \alpha_1, \cdots, \alpha_r\}$ of $A$, the map
$$f|_{X_{\alpha_1} \cap \cdots \cap X_{\alpha_r}} : X_{\alpha_1} \cap \cdots \cap X_{\alpha_r} \rightarrow Y_{\alpha_1} \cap \cdots \cap Y_{\alpha_r}$$
is a homotopy equivalence. Then the map $f$ is a homotopy equivalence.
\end{thm}
\begin{proof}
This is well-known. See Theorem 2.4 of \cite{Matsushita 2}, for example.
\end{proof}

\begin{thm}[Theorem 2.4 of \cite{GM}] \label{Bredon}
Let $\Gamma$ be a finite group, and $f:X \rightarrow Y$ a $\Gamma$-map between free $\Gamma$-CW-complexes. Then $f$ is a $\Gamma$-homotopy equivalence if and only if $f$ is a homotopy equivalence.
\end{thm}

\section{Deformations of box complexes}

In this section we shall show that a certain subdivision of a graph does not change the homotopy type of $\Hom(K_2,G)$. We write $B(G)$ instead of $\Hom(K_2,G)$ for short, and call it the {\it box complex of $G$} \cite{MZ}.

For a positive integer $n$, define the graph $L_n$ by $V(L_n) = \{ 0,1,\cdots, n\}$ and $E(L_n) = \{ (a,b) \; | \; |a-b| =1\}$.

\vspace{1mm}
Let $G$ be a graph and $e = \langle v,w \rangle$ an edge of $G$. We define the graph $G_e$ as follows. The vertex set of $G_e$ is $V(G) \coprod \{ 0,1\}$ and the edge set is defined by
$$E(G_e) = (E(G) \setminus e) \cup \langle 0,1\rangle \cup \langle 0,v\rangle \cup \langle 1,w \rangle.$$
Figure 1 illustrates the graph $G_e$. We have a natural homomorphism $r_e :G_e \rightarrow G$ defined by the correspondence
$$r_e(x) = \begin{cases}
x & (x \in V(G))\\
w & (x = 0)\\
v & (x = 1).
\end{cases}$$

\begin{center}
\begin{picture}(295,100)(0,0)
\put(30,40){\circle*{4}} \put(30,40){\line(-1,2) {10}} \put(30,40){\line(-2,1){20}} \put(30,40){\line(-2,-1){20}} \put(30,40){\line(0,-1){20}}

\put(90,60){\circle*{4}} \put(30,40){\line(3,1){60}}

\put(90,60){\line(1,3){7}} \put(90,60){\line(3,1){20}} \put(90,60){\line(2,-1){20}} \put(90,60){\line(1,-3){7}}

\put(200,40){\circle*{4}} \put(200,40){\line(-1,2) {10}} \put(200,40){\line(-2,1){20}} \put(200,40){\line(-2,-1){20}} \put(200,40){\line(0,-1){20}}

\put(200,40){\line(2,1){20}} \put(220,50){\line(1,0){25}} \put(245,50){\line(2,1){20}}

\put(265,60){\line(1,3){7}} \put(265,60){\line(3,1){20}} \put(265,60){\line(2,-1){20}} \put(265,60){\line(1,-3){7}}

\put(220,50){\circle*{4}} \put(245,50){\circle*{4}} \put(265,60){\circle*{4}}

\put(130,50){\vector(1,0){30}}

\put(30,45){$v$} \put(82,65){$w$} \put(61,42){$e$}

\put(200,45){$v$} \put(257,65){$w$} \put(218,38){$0$} \put(243,38){$1$}

\put(58,10){$G$} \put(230,10){$G_e$}
\end{picture}

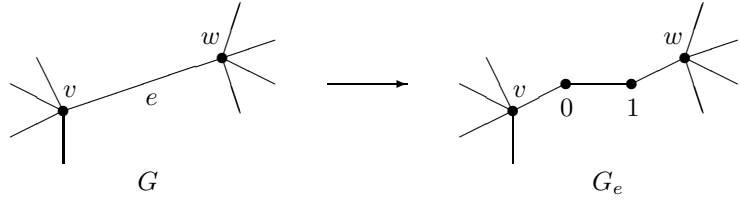
{\bf Figure 1.}
\end{center}

In general $r_e:G_e \rightarrow G$ does not induce a homotopy equivalence between their box complexes. For example, consider the cycles $C_6$ and $C_4$. On the other hand, the following proposition holds.

\begin{prop} \label{Proposition 3.2}
Let $G$ be a graph, and $e = \langle v, w \rangle$ an edge of $G$. If there is no graph homomorphism from $L_3 \rightarrow G \setminus e$ which takes $0$ to $v$ and $3$ to $w$, then the map $r_e: G_e \rightarrow G$ induces a homotopy equivalence $r_{e*} : B(G_e) \rightarrow B(G)$. Moreover, $B(G)$ is obtained from $B(G_e)$ by collapsing two intervals to two points, respectively. 
\end{prop}
\begin{proof}
First we describe the inverse image of each point of $B(G)$ with respect to $r_{e*}$. We note that $r_e^{-1}(x) = \{ x\}$ if $x \neq v,w$, and $r_e^{-1}(v) = \{ v,1\}$, $r_e^{-1}(w) = \{ w, 0\}$.

Let $(\sigma,\tau) \in B(G)$ and let $(\sigma' , \tau') \in r_{e*}^{-1}(\sigma, \tau)$. Suppose that $\sigma$ contains neither $v$ nor $w$. Let $x \in \sigma$. Then we have $x \in \sigma'$. Since $0$ and $1$ are not adjacent to $x$ in $G_e$, $\tau'$ contains neither $0$ nor $1$. Thus we have $\tau' = \tau$. Similarly, if $\tau$ contains neither $v$ nor $w$, then we have $r_{e*}^{-1}(\sigma, \tau) = \{ (\sigma,\tau)\}$.

Next we consider the following cases.
\begin{itemize}
\item[(a)] $v \in \sigma$, $w \in \tau$, and there is $x \in \sigma$ such that $x \neq v$.
\item[(b)] $v \in \sigma$, $w \in \tau$, and there is $x \in \tau$ such that $x \neq w$.
\item[(c)] $w \in \sigma$, $v \in \tau$, and there is $x \in \sigma$ such that $x \neq w$.
\item[(d)] $w \in \sigma$, $v \in \tau$, and there is $x \in \tau$ such that $x \neq v$.
\end{itemize}
We first note that in the case (a), we have $\tau = \{ w\}$. In fact, if there is $y \in \tau$ such that $y \neq w$, then we have a graph homomorphism $f: L_3 \rightarrow G$ defined by
$$f(0) = v, f(1) = y, f(2) = x, f(3) = w.$$
This contradicts the hypothesis. Similarly, we have $\sigma = \{ v\}$ in the case (b), $\tau = \{ v\}$ in the case (c), and $\sigma = \{ w \}$ in the case (d). This means that an element of $B(G)$ satisfies at most one of the above four conditions. Moreover, if an element $(\sigma ,\tau)$ of $B(G)$ satisfies the condition ($i$) for some $i \in \{ 1,\cdots, 4\}$, then an element of $B(G)$ smaller than $(\sigma,\tau)$ (for the definition of the partial order of $B(G)$, see Section 2.2) does not satisfy the other conditions described above.

We now turn to the description of the inverse image of $r_{e*}$. Suppose that $(\sigma, \tau)$ satisfies the condition (a). Then we have $\tau' = \tau = \{ w\}$. Since $w \in \tau'$ and $v$ is not adjacent to $w$ in $G_e$, we have that $v \not\in \sigma'$ and $1 \in \sigma'$. Thus we have
$$\sigma' = (\sigma \setminus \{ v\}) \cup \{ 1\}.$$
The cases of (b), (c), and (d) are similarly described.

Finally, we set $I = r_{e*}^{-1}(\{ v\}, \{ w\})$ and $J = r_{e*}^{-1}(\{ w\},\{ v\})$. Clearly, we have
$$I = \{ (\{ v\}, \{ 0\}), (\{ v, 1\}, \{ 0\}), (\{ 1\}, \{ 0\}), (\{ 1\}, \{ 0, w\}), (\{ 1\}, \{ w\}) \}.$$
Figure 2 is the Hasse diagram of $I$. Thus its classifying space $|I|$ is an interval. Similarly, the classifying space $|J|$ of $J$ is an interval.

\begin{center}
\begin{figure}[b]
\begin{picture}(240,75)
\put(0,20){\circle*{3}} \put(60,50){\circle*{3}} \put(120,20){\circle*{3}} \put(180,50){\circle*{3}} \put(240,20){\circle*{3}}

\put(0,20){\line(2,1){60}} \put(60,50){\line(2,-1){60}} \put(120,20){\line(2,1){60}} \put(180,50){\line(2,-1){60}}

\put(-23,7){ $(\{ v\}, \{ 0\})$} \put(30,60){$(\{ v, 1\}, \{ 0\})$} \put(99,7){$(\{ 1\}, \{ 0\})$} \put(150,60){$(\{ 1\}, \{ 0,w\})$} \put(217,7){$(\{ 1\}, \{ w\})$}
\end{picture}

The Hasse diagram of $I$.

\vspace{2mm} {\bf Figure 2.}
\end{figure}
\end{center}

Let $X$ be a topological spaces obtained from $|B(G_e)|$ by collapsing $|I|$ and $|J|$, respectively. Let $q: |B(G_e)| \rightarrow X$ be the quotient map. Since $|I| \cap |J| = \emptyset$ (see the definitions of $I$ and $J$), we have that $q$ is a homotopy equivalence. Clearly, the map $|r_{e*}| : |B(G_e)| \rightarrow |B(G)|$ induced by $r_e$ induces a continuous map $f: X \rightarrow |B(G)|$.

Before constructing the inverse of $f$, we prepare some notation. Let $\R^{(B(G))}$ denote the free $\R$-module generated by $B(G)$, and we consider its topology as the direct limit of the finite dimensional $\R$-submodules. For each element $x \in B(G)$, we write $e_x$ to indicate the associated element of $\R^{(B(G))}$. For each chain $c = \{ (\sigma_0, \tau_0) , \cdots , (\sigma_n, \tau_n) \; | \; (\sigma_i, \tau_i) < (\sigma_j, \tau_j)$ if $i < j\}$ of $B(G)$, we write $\Delta_c$ to indicate the $n$-simplex associated to the chain $c$. Namely, $\Delta_c$ is the convex hull of $e_{(\sigma_0,\tau_0)}, \cdots, e_{(\sigma_n, \tau_n)}$ in $\R^{(B(G))}$. To construct the inverse $g: |B(G)| \rightarrow X$ of $f : X \rightarrow |B(G)|$, we construct $g_c: \Delta_c \rightarrow X$ for each chain $c$ of $B(G)$.

Let $c = \{ (\sigma_0, \tau_0), \cdots, (\sigma_n, \tau_n) \; | \; (\sigma_0,\tau_0) < \cdots < (\sigma_n, \tau_n)\}$ be a chain of $B(G)$. If $\sigma_n$ contains neither $v$ nor $w$, then $c$ is a chain of $B(G_e)$. Thus define $g_c$ by the composition of the sequence
$$\begin{CD}
\Delta_c \subset |B(G_e)| \longrightarrow X.
\end{CD}$$
The case that $\tau$ contains neither $v$ nor $w$ is similar, and these definitions of $g_c$ coincide if both of $\sigma$ and $\tau$ contain neither $v$ nor $w$.

Next suppose that $(\sigma_n, \tau_n)$ satisfies the condition (a) mentioned above. Set
$$\sigma_i' = \begin{cases}
(\sigma_n \setminus \{ v\}) \cup \{ 1\} & (v \in \sigma_i)\\
\sigma_i & (v \not\in \sigma_i)
\end{cases}$$
and set 
$$c' = \{ (\sigma'_0, \tau_0) , \cdots , (\sigma_n' , \tau_n)\}.$$
Then $c'$ is a chain in $B(G_e)$ and define $\iota : \Delta_c \rightarrow \Delta_{c'}$ be the simplicial map which sends $(\sigma_i, \tau_i)$ to $(\sigma_i', \tau_i)$. Define $g_c : \Delta_c \rightarrow X$ by the composition of the sequence
$$\begin{CD}
\Delta_c @>{\iota}>> \Delta_{c'} \subset |B(G)| @>>> X.
\end{CD}$$
The cases (b), (c), and (d) can be similarly proved.

If $c = \{ (\{ v\}, \{ w\})\}$ (or $c = \{ (\{ w \}, \{ v\})\}$), then define $g_c : \Delta_c \rightarrow X$ by sending the point $e_{(\{ v\}, \{ w\})}$ of $\Delta_c$ to the point of $X$ obtained by collapsing $|I|$ (or $|J|$, respectively).

We want to show that the maps $g_c$ define a continuous map $g: |B(G)| \rightarrow X$. Let $c$ and $c'$ be chains of $B(G)$ and suppose $c \subset c'$. Then it suffices to show that
\begin{eqnarray} \label{diagram}
\begin{CD}
\Delta_c @>{g_c}>> X\\
@VVV @|\\
\Delta_{c'} @>{g_{c'}}>> X 
\end{CD}
\end{eqnarray}
is commutative. Have the left vertical arrow is the inclusion. Let $(\sigma,\tau)$ (or $(\sigma' , \tau')$) be the maximal element of $c$ (or $c'$, respectively). If $\sigma'$ contains neither $v$ nor $w$, then neither $\sigma$ and it is clear that the above commutative diagram is commutative. The case $\tau'$ contains neither $v$ nor $w$ is similarly proved.

If $(\sigma',\tau')$ satisfies the condition (a) described as above. Then $(\sigma, \tau)$ satisfies (1), $\sigma$ contains neither $v$ nor $w$, or $(\sigma, \tau) = (\{ v\}, \{ w\})$ (see the remark given in the paragraph after the condition (d)). It is straightforward to see that the diagram (\ref{diagram}) is again commutative. The cases of (b), (c), and (d) are similarly proved.

It is obvious that the diagram (\ref{diagram}) is commutative when $(\sigma', \tau') = (\{ v\}, \{ w\})$ or $(\{ w\}, \{ v\})$, since in this case $c'$ is a minimal element of the poset of chains of $B(G)$. Thus the maps $g_c$ induce a continuous map $g: |B(G)| \rightarrow X$.

By the description of the inverse image of $r_{e*} : B(G_e) \rightarrow B(G)$, it is clear that $g$ is the inverse of $f$. This completes the proof.
\end{proof}

Note that if the girth $g(G)$ of a graph $G$ is greater than 4, then the hypothesis of Proposition \ref{Proposition 3.2} always holds.

\begin{eg}[Walker \cite{Walker}] \label{Example 3.3}
Let $G_1$ and $G_2$ be graphs illustrated in Figure 3. Clearly, they have different chromatic numbers. On the other hand, the box complexes $B(G_1)$ and $B(G_2)$ are $\Z_2$-homotopy equivalent. This fact can be deduced from the direct verification or Proposition \ref{Proposition 3.2}.

\begin{center}
\begin{figure}[h]
\begin{picture}(170,60)
\put(10,10){\circle*{3}} \put(10,30){\circle*{3}} \put(30,30){\circle*{3}} \put(40,50){\circle*{3}}
\put(10,10){\line(1,1){20}}  \put(10,10){\line(0,1){20}} \put(10,30){\line(1,0){20}} \put(10,30){\line(3,2){30}} \put(30,30){\line(1,2){10}}
\put(50,30){\circle*{3}} \put(70,30){\circle*{3}} \put(70,10){\circle*{3}}
\put(10,10){\line(1,0){60}} \put(50,30){\line(1,-1){20}} \put(50,30){\line(1,0){20}} \put(40,50){\line(1,-2){10}} \put(40,50){\line(3,-2){30}}
\put(70,10){\line(0,1){20}}

\put(35,-3){$G_1$}

\put(110,10){\circle*{3}} \put(110,30){\circle*{3}} \put(130,30){\circle*{3}} \put(140,50){\circle*{3}}
\put(110,10){\line(1,1){20}}  \put(110,10){\line(0,1){20}} \put(110,30){\line(1,0){20}} \put(110,30){\line(3,2){30}} \put(130,30){\line(1,2){10}}
\put(150,30){\circle*{3}} \put(170,30){\circle*{3}} \put(170,10){\circle*{3}}
\put(110,10){\line(1,0){60}} \put(150,30){\line(1,-1){20}} \put(150,30){\line(1,0){20}} \put(140,50){\line(1,-2){10}} \put(140,50){\line(3,-2){30}}
\put(170,10){\line(0,1){20}} \put(130,10){\circle*{3}} \put(150,10){\circle*{3}}

\put(135,-3){$G_2$}
\end{picture}

\vspace{2mm}{\bf Figure 3.}
\end{figure}
\end{center}
\end{eg}

\section{Bipartite graphs}

Throughout this section, $T$ is a connected bipartite graph having at least one edge. The goal of this section is to prove Proposition \ref{Proposition 4.5}, which asserts that if the girth of $G$ is sufficiently large, then the subdivision considered in Proposition \ref{Proposition 3.2} does not change the homotopy type of $\Hom(T,G)$.

\begin{prop}\label{Proposition 4.1}
Let $X$ be a finite tree having at least one edge, let $T$ be a finite connected graph with $\chi(T) = 2$, and let $u : K_2 \rightarrow T$ be a graph homomorphism. Then $u^* : \Hom(T,X) \rightarrow \Hom(K_2,X)$ is a homotopy equivalence.
\end{prop}
\begin{proof}
We prove this by induction of the cardinality of $V(X)$. Suppose that $\# V(X) = 2$, namely, $X = K_2$. Since $\Hom(T,K_2) \cong S^0$ and $\Hom(K_2, K_2) \cong S^0$, it is clear that $u^* : \Hom(T, K_2) \rightarrow \Hom(T, X)$ is a homotopy equivalence. Suppose that $\# V(X) > 2$. Then $X$ has a dismantlable vertex $v$ (a leaf of $X$), and consider the following commutative diagram:
$$\begin{CD}
\Hom(T,X \setminus v) @>>> \Hom(T,X)\\
@V{u^*}VV @VV{u^*}V\\
\Hom(K_2, X \setminus v) @>>> \Hom(K_2,X)
\end{CD}$$
The horizontal arrows are the maps induced by the inclusion $X \setminus v \hookrightarrow X$. Then Proposition \ref{Proposition 2.3} implies that the horizontal arrows are homotopy equivalences. On the other hand, the inductive hypothesis implies that the left vertical arrow is a homotopy equivalence. Thus the right vertical arrow is a homotopy equivalence.
\end{proof}

\begin{lem} \label{Lemma 4.4}
Let $T$ be a finite connected bipartite graph with positive diameter $\Delta$, and $X$ a finite graph whose girth is greater than $2 \Delta + 4$. Then the map
$$u^* : \Hom(T,X) \rightarrow \Hom(K_2,X)$$
induced by a graph homomorphism $u : K_2 \rightarrow T$  is a homotopy equivalence.
\end{lem}
\begin{proof}
Let $\mathcal{Y}$ be the family of connected subgraphs of $X$ whose diameter is smaller than $\Delta +2$. Then every multi-homomorphism $\eta$ from $T$ to $X$ factors through $Y$ for some $Y \in \mathcal{F}$. Thus we have 
$$\Hom(T,X) = \bigcup_{Y \in \mathcal{Y}} \Hom(T, Y).$$
Similarly we have
$$\Hom(K_2,X) = \bigcup_{Y \in \mathcal{Y}} \Hom(K_2, Y).$$
Thus to see that $u^* : \Hom(T,X) \rightarrow \Hom(K_2, X)$ is a homotopy equivalence, it suffices to see the following assertion by Theorem \ref{Theorem 2.7}: For a positive integer $r$ and a finite subset $\{ Y_1, \cdots, Y_r\}$ of $\mathcal{Y}$, the map
\begin{eqnarray} \label{eqn 4.1}
u^* : \bigcap_{i = 1}^r \Hom(T, Y_i) \rightarrow \bigcap_{i =1}^r \Hom(K_2, Y_i)
\end{eqnarray}
is a homotopy equivalence.

We first note that
$$\bigcap_{i = 1}^r \Hom(T, Y_i) = \Hom \Big(T, \bigcap_{i=1}^r Y_i \Big) , \; \bigcap_{i=1}^r \Hom(K_2, Y_i) = \Hom \Big( K_2, \bigcap_{i=1}^r Y_i \Big)$$
Thus if every vertex of $Y_1 \cap \cdots \cap Y_r$ is isolated, then the both sides of the map (\ref{eqn 4.1}) are empty and hence the map is a homotopy equivalence.

On the other hand, suppose that $Y_1 \cap \cdots \cap Y_r$ has an edge. By Proposition \ref{Proposition 4.1}, it suffices to show that $Y_1 \cap \cdots \cap Y_r$ is a tree. Note that $Y_1$ has no embedded cycle since its diameter is smaller than $\Delta + 2$ and the girth of $X$ is greater than $2 \Delta + 4$. Therefore $Y_1 \cap \cdots \cap Y_r$ has no embedded cycles. Thus it suffices to see that $Y_1 \cap \cdots \cap Y_r$ is connected.

Before giving the proof, we note the following: Let $x$ and $y$ be vertices of $X$. We call a graph homomorphism $\varphi : L_n \rightarrow X$ a {\it path joining $x$ to $y$} if $\varphi (0) = x$ and $\varphi(n) = y$. Moreover, if there is no $i \in \{ 1,\cdots, n -1\}$ such that $\varphi(i -1) = \varphi (i+1)$, we call the path $\varphi$ {\it non-degenerate}. It is straightforward to show that a non-degenerate path $\varphi : L_n \rightarrow X$ joining $x$ to $y$ with $n \leq \Delta + 1$ is unique since the girth of $X$ is greater than $2 \Delta +4$.

We now turn to the proof that $Y_1 \cap \cdots \cap Y_r$ is connected. Suppose that $Y_1 \cap \cdots \cap Y_r$ is not empty and let $x, y \in V(Y_1 \cap \cdots \cap Y_r)$. By the definition of $\mathcal{Y}$, $Y_i$ is connected and hence is a tree for each $i = 1,\cdots, r$. Therefore let $\varphi_i : L_{n_i} \rightarrow Y_i$ be the shortest path joining $x$ with $y$ in $Y_i$. Since $Y_i$ is a tree whose diameter is smaller than $\Delta + 2$, we have that $n_i < \Delta +2$. Thus it follows from the previous paragraph that $\varphi_1 = \cdots = \varphi_r$. This implies that $x$ and $y$ belong to the same connected component of $Y_1 \cap \cdots \cap Y_r$.
\end{proof}

\begin{prop} \label{Proposition 4.5}
Let $T$ be a finite connected bipartite graph with positive diameter $\Delta$, and $X$ a finite graph with girth greater than $2 \Delta + 4$. For each edge $e$ of $X$, the map $r_e : X_e \rightarrow X$ induces a homotopy equivalence $r_{e*} : \Hom(T,X_e) \rightarrow \Hom(T,X)$.
\end{prop}
\begin{proof}
Consider the commutative diagram
$$\begin{CD}
\Hom(T,X_e) @>{r_{e*}}>> \Hom(T,X)\\
@VVV @VVV\\
\Hom(K_2, X_e) @>{r_{e*}}>> \Hom(K_2,X).
\end{CD}$$
The vertical arrows are homotopy equivalences by Lemma \ref{Lemma 4.4}, and the lower horizontal arrow is a homotopy equivalence by Proposition \ref{Proposition 3.2}.
\end{proof}

\section{Proof of the main theorem}

The purpose of this section is to complete the proof of Theorem 1.1. We actually prove Theorem \ref{Theorem 5.1} and Corollary \ref{Corollary 5.5}, which are a little generalized assertions.

Let $\mathcal{F}$ be a (not necessarily small) family of finite connected graphs and suppose that there is a positive integer $m$ which satisfies the following conditions:

\begin{itemize}
\item The diameter of a graph belonging to $\mathcal{F}$ is smaller than $m$.
\item If $T \in \mathcal{F}$ is not bipartite, then the odd girth of $T$ is smaller than $m$.
\end{itemize}
In this case, we call a family $\mathcal{F}$ {\it uniformly small}. Note that if $\mathcal{F}$ is a finite family of finite graphs, then $\mathcal{F}$ is uniformly small. Therefore the main theorem (Theorem \ref{Theorem 1.1}) is deduced from the following theorem and Corollary \ref{Corollary 5.5} since $\Hom(T_1 \coprod T_2, G) \cong \Hom(T_1,G) \times \Hom(T_2, G)$.

\begin{thm}\label{Theorem 5.1}
Let $\mathcal{F}$ be a uniformly small family of graphs and let $G$ be a non-bipartite graph. Then for every positive integer $n$, there is an inclusion $f:G \hookrightarrow H$ such that $f$ induces a homotopy equivalence $\Hom(T,G) \rightarrow \Hom(T,H)$ for every $T \in \mathcal{F}$ but $\chi(H) > n$. Moreover, if $G$ is finite and connected, then one can take $H$ to be finite and connected.
\end{thm}

To prove this theorem, we recall some facts of homotopy pushouts. Here we consider that {\it all maps between CW-complexes are cellular}, for simplicity. For a more general treatment, we refer to \cite{WZZ}.

Let $f: X \rightarrow Y$ and $g: X \rightarrow Z$ be cellular maps. The homotopy pushout of $f$ and $g$ is the space
$$E(f,g) = ( X \times [0,1] ) \coprod Y \coprod Z / \sim,$$
where the equivalence relation is generated by $(x,0) \sim f(x)$ and $(x,1) \sim g(x)$ for every $x \in X$.

\begin{lem}[Proposition 3.7 of \cite{WZZ}] \label{Lemma 5.2}
Consider a commutative diagram
$$\begin{CD}
Y @<f<< X @>g>> Z\\
@VVV @VVV @VVV\\
Y' @<<< X' @>>> Z'
\end{CD}$$
of CW-complexes. If the all vertical arrows are homotopy equivalences, then the map $E(f,g) \rightarrow E(f',g')$ is a homotopy equivalence.
\end{lem}

A commutative square
\begin{eqnarray} \label{diagram 5}
\begin{CD}
X @>f>> Y\\
@VgVV @VVV\\
Z @>>> W
\end{CD}
\end{eqnarray}
of CW-complexes is a {\it homotopy pushout square} if the map $E(f,g) \rightarrow W$ is a homotopy equivalence.

\begin{lem}[Lemma 2.4 of \cite{WZZ}] \label{Lemma 5.3}
If the diagram (\ref{diagram 5}) is a pushout diagram and either $f$ or $g$ is an inclusion, then the diagram (\ref{diagram 5}) is a homotopy pushout square.
\end{lem}

\begin{lem} \label{Lemma 5.4}
Let $f: X \rightarrow Y$ and $g: X \rightarrow Z$ be cellular maps. If $f$ is a homotopy equivalence, then the inclusion $Z \hookrightarrow E(f,g)$ is a homotopy equivalence.
\end{lem}
\begin{proof}
Consider the commutative diagram
$$\begin{CD}
X @= X @>g>> Z\\
@VfVV @| @|\\
Y @<f<< X @>g>> Z.
\end{CD}$$
Then Lemma \ref{Lemma 5.2} implies that $E({\rm id}_X , g) \rightarrow E(f,g)$ is a homotopy equivalence. Clearly, the inclusion $Z \hookrightarrow E({\rm id}_X,Z)$ is a homotopy equivalence. Thus the composition $Z \hookrightarrow E({\rm id}_X, g) \rightarrow E(f,g)$ is a homotopy equivalence.
\end{proof}

Now we turn to the proof of Theorem 5.1.

\vspace{2mm}
\noindent {\it Proof of Theorem 5.1.} 
Let $m$ be a positive integer which satisfies the conditions (1) and (2) in the beginning of this section. Let $X$ be a finite connected graph such that $\chi(X) > n$ and $g(X) \geq m$.  Since $G$ is not bipartite, there is an integer $k$ such that there is a graph homomorphism $C_{2k+1} \rightarrow G$. Let $Y$ be the graph obtained by replacing each edge of $X$ by the line $L_{2k+1}$ with length $2k+1$. It is clear that there is a graph homomorphism from $Y$ to $C_{2k+1}$, and hence there is a graph homomorphism $g: Y \rightarrow G$. Since $Y$ is obtained by iterating subdivisions considered in Proposition \ref{Proposition 3.2}, it follows from Proposition \ref{Proposition 4.5} that there is a graph homomorphism $f:Y \rightarrow X$ which induces a homotopy equivalence $\Hom(T,Y) \rightarrow \Hom(T,X)$ if $T$ is bipartite.

Let $H$ be the colimit of the diagram
$$\begin{CD}
X @<f<< Y @>{\iota_0}>> Y \times I_m @<{\iota_m}<< Y @>g>> G,
\end{CD}$$
where $\iota_j : Y \rightarrow Y \times I_m$ is defined by the correspondence $y \mapsto (y,j)$. Roughly speaking, $H$ is obtained by attaching the ``ends of the cylinder $Y \times I_m$" by the graph homomorphisms $f:Y \rightarrow X$ and $g: Y \rightarrow G$. Let $A = X \cup_Y Y \times I_{m-1}$ and $B = Y \times I_{[1,m]} \cup_Y G$ and consider $A$ and $B$ as subgraphs of $H$.

Since the diameter of $T$ is smaller than $(m-2)$, every multi-homomorphism from $T$ to $H$ factors through $A$ or $B$. Thus we have
$$\Hom(T,H) = \Hom(T,A) \cup \Hom(T,B).$$

Suppose that $T$ is non-bipartite. Since the odd girth of $T$ is smaller than the girth of $X$, there is no graph homomorphism from $T$ to $X$. Hence there is no graph homomorphism from $T$ to $A$ since there is a graph homomorphism from $A$ to $X$. Thus we have $\Hom(T,A) = \emptyset$ and $\Hom(T,H) = \Hom(T,B)$. Since the inclusion $G \hookrightarrow B$ is a $\times$-homotopy equivalence, we have that $\Hom(T,G) \hookrightarrow \Hom(T,B) = \Hom(T,H)$ is a homotopy equivalence. This completes the proof in case $T$ is not bipartite.

Next suppose that $T$ is bipartite. Consider the following diagram
$$\begin{CD}
|\Hom(T,A)| @<<< |\Hom(T,A \cap B)| @>>> |\Hom(T,B)|\\
@VVV @VVV @VVV\\
|\Hom(T,X)| @<{\simeq}<< |\Hom(T,Y)| @>>> |\Hom(T,G)|.
\end{CD}$$
Since the graph homomorphisms $A \rightarrow X$, $A \cap B \cong Y \times I_{[1,m-1]} \rightarrow Y$, and $B \rightarrow G$ are $\times$-homotopy equivalences, we have that the vertical arrows in the above diagram are homotopy equivalences. Let $E$ be the homotopy pushout of the upper horizontal arrows and $E'$ the homotopy pushout of the lower horizontal arrows. Since the upper horizontal arrows are inclusions, the natural map $E \rightarrow |\Hom(T,H)|$ is a homotopy equivalence (Lemma \ref{Lemma 5.3}). Then we have the commutative diagram
$$\begin{CD}
|\Hom(T,G)| @>{\simeq}>> E'\\
@V{\simeq}VV @AA{\simeq}A \\
|\Hom(T,B)| @>>> E @>{\simeq}>> |\Hom(T,H)|.\\
\end{CD}$$
Since the inclusion $G \hookrightarrow B$ is a $\times$-homotopy equivalence, the left vertical arrow is a homotopy equivalence. Other homotopy equivalences are deduced from Lemma \ref{Lemma 5.2} and Lemma \ref{Lemma 5.4}. Thus we have that the composition $|\Hom(T,G)| \hookrightarrow |\Hom(T,B)| \hookrightarrow E \rightarrow |\Hom(T,H)|$ is a homotopy equivalence.
\qed

\vspace{2mm}
We conclude this section with the following corollary.

\begin{cor}\label{Corollary 5.5}
Let $\mathcal{F}$ be a uniformly small family of flipping $\Z_2$-graphs. Then for every graph $G$ with $\chi(G) \geq 3$ and for every positive integer $m$, there is an inclusion $f:G \hookrightarrow H$ such that $f$ induces a $\Z_2$-homotopy equivalence $\Hom(T,G) \rightarrow \Hom(T,H)$ for every $T \in \mathcal{F}$ but $\chi(H) > m$. Moreover, if $G$ is finite and connected, then we can take $H$ to be finite and connected.
\end{cor}
\begin{proof}
If $G$ has a looped vertex then put $H = G$. If $G$ has no looped vertex, then the graph $H$ constructed in the proof of Theorem \ref{Theorem 5.1} has the desired properties. Indeed, since $G$ and $H$ are simple and $T \in \mathcal{F}$ is a flipping $\Z_2$-graph, we have that $\Hom(T,G)$ and $\Hom(T,H)$ are free $\Z_2$-complexes. Thus Theorem \ref{Bredon} implies that $i_* : \Hom(T,G) \rightarrow \Hom(T,H)$ is a $\Z_2$-homotopy equivalence.
\end{proof}

\end{document}